\documentclass[final,3p,mathptmx]{elsarticle}

\usepackage{amsmath,amsthm,amssymb,amsfonts}
\usepackage{color}
\usepackage{caption}

\newcommand{\supp}{\mathop{\mathrm{supp}}}
\newcommand{\lspan}{\mathop{\mathrm{span}}}
\newcommand{\diag}{\mathop{\mathrm{diag}}}

\theoremstyle{plain}
\newtheorem{thm}{Theorem}[section]
\newtheorem{lem}[thm]{Lemma}
\newtheorem{prop}[thm]{Proposition}

\newtheorem{conj}[thm]{Conjecture}

\theoremstyle{definition}
\newtheorem{defn}{Definition}[section]

\theoremstyle{remark}
\newtheorem{rem}[thm]{Remark}

\everymath{\displaystyle}

\begin{document}

\title{Semi-regular Dubuc-Deslauriers wavelet tight frames}
\author{Alberto Viscardi}
\ead{alberto.viscardi@unimib.it}
\address{Dipartimento di Matematica e Applicazioni, Universit\`a degli Studi di Milano-Bicocca, via Roberto Cozzi 55, 20126 Milano, Italy}
\date{}

\begin{abstract}
In this paper, we construct wavelet tight frames with $n$ vanishing moments for Dubuc-Deslauriers $2n$-point
semi-regular interpolatory subdivision schemes. Our motivation for this construction is its practical use
for further regularity analysis of wide classes of semi-regular subdivision. Our constructive tools are
local eigenvalue convergence analysis for semi-regular Dubuc-Deslauriers subdivision,
the Unitary Extension Principle and the generalization of the Oblique Extension Principle to the
irregular setting by Chui, He and St\"ockler.  This group of authors derives suitable approximation of the inverse 
Gramian for irregular B-spline 
subdivision. Our main contribution is the derivation of the appropriate 
approximation of the inverse Gramian for the semi-regular Dubuc-Deslauriers scaling functions 
ensuring $n$ vanishing moments of the corresponding framelets. 
\end{abstract}

\begin{keyword}
wavelet \sep tight frame \sep semi-regular \sep subdivision \sep Dubuc-Deslauriers \sep UEP
\end{keyword}

\maketitle

\section{Introduction}

Wavelets, from the seminal works of Meyer \cite{MR1228209} and Daubechies \cite{MR1162107}, and wavelet tight frames, from the works of Ron and Shen \cite{MR1491938,MR1469348} and further contributions in \cite{MR1968118,MR1971300}, have been a productive research area in the last thirty years, both in theory and applications (see e.g. \cite{MR2479996}).
One of the standard starting points in the construction of such function systems are refineable functions arising from subdivision schemes, iterative methods for generating curves and surfaces (see e.g. \cite{MR2683008,MR2415757,MR2723195,Warren:2001:SMG:580358}). In the stationary uniform setting, i.e. when the subdivision rules are shift-invariant and do not change between the iterations, the so called \emph{Unitary Extension Principle} (UEP) \cite{MR1491938,MR1469348} and \emph{Oblique Extension Principle} (OEP) from \cite{MR1968118,MR1971300} are used for constructions of wavelet tight frames with one or more vanishing moments. UEP and OEP are based on Fourier techniques and on factorizations of trigonometric polynomials. A generalization of the UEP procedure for nonstationary uniform schemes, when the subdivision rules can change from one iteration to the other, was presented in \cite{MR2443258}. 
A general setting that also covers the semi-regular case is the one proposed in \cite{MR2082157,MR2110512} where a matrix formulation of the UEP and OEP is given and examples with nonuniform B-spline schemes are shown.

The main contributions of this paper are twofold: we construct wavelet tight frames with $n$ vanishing moments from semi-regular Dubuc-Deslauriers $2n$-point subdivision schemes and, to meet this goal, we also present a detailed convergence analysis of such semi-regular schemes.
The family of semi-regular Dubuc-Deslauriers $2n$-point schemes was defined in \cite{MR1356994}, where, without loss of generality, the initial mesh is of the type 
\begin{equation} \label{eq:t}
\mathbf{t}(k)\;=\;\left\{
\begin{array}{rl}
kh_\ell & \textrm{if } k<0\\ \\
kh_r & \textrm{if } k\geq0
\end{array}
\right.\;,\quad h_\ell,\;h_r\;>\;0\;.
\end{equation}
Our convergence analysis of this family uses the local eigenvalue analysis \cite{MR2415757,Warren:2001:SMG:580358}.
Our construction of the corresponding wavelet tight frames on the regular part of the mesh uses the UEP and is based on the well known \cite{MR1397613} link between Dubuc-Deslauriers and Daubechies scaling functions \cite{MR1162107}. The interpolation and polynomial generation (up to degree $2n-1$) properties of the corresponding subdivision schemes ensure $n$ vanishing moments for the framelets. 
On the irregular part of the mesh, in a neighborhood of $\mathbf{t}(0)$, we apply the matrix factorization technique from \cite{MR2082157,MR2110512}, a generalization of the Oblique Extension Principle (OEP) \cite{MR1968118,MR1971300} to the irregular setting. Similarly to \cite{MR2419706}, instead of factorizing a certain global positive semi-definite matrix, we used the regular framelets to reduce the construction on the irregular part of the mesh to a factorization of a $(12n-9)\times(12n-9)$ matrix. The corresponding vanishing moment recovery matrix is a suitable approximation to the inverse Gramian.
This approximation to the inverse Gramian guarantees $n$ vanishing moments of the irregular framelets. However, the existence of the underlying scaling functions is ensured only for certain values of $h_\ell/h_r$.

The advantage of our construction is in its simplicity. Indeed, with our UEP based construction we obtain regular framelets with $n$ vanishing moments without endeavouring into more tedious computations required by the OEP. Furthermore, compared to the B-spline based wavelet tight frames in \cite{MR1968118} (the only other semi-regular wavelet tight frame in literature) whose filters have size (number of non-zero coefficients) $3n-1$, the corresponding filters obtained from the Dubuc-Deslauriers $2n$-point framelets are of size $2n+1$. The disadvantage occurs on the irregular part of the mesh, where our filters have possibly larger supports.

The paper is organized as follows. Section \ref{sec:notes} presents notation and recalls some background facts about subdivision and wavelet tight frames. In Section \ref{sec:DD} we introduce the family of semi-regular Dubuc-Deslauriers $2n$-point schemes, providing their convergence analysis, and define the corresponding scaling functions. The wavelet tight frame construction for the regular case is presented in Section \ref{sec:UEP}. In Section \ref{sec:semi_reg}, we compute the Gramian and present our semi-regular wavelet tight frame construction. To conclude, in Section \ref{sec:examples}, we illustrate our theoretical results by examples for $n=1,2$.

\section{Notation and Background} \label{sec:notes}

\noindent In this section, we recall some basic facts about subdivision and wavelet tight frames.

\noindent A stationary subdivision scheme with the bi-infinite subdivision matrix $\mathbf{P}: \ell(\mathbb{Z}) \rightarrow \ell(\mathbb{Z})$ 
maps recursively the initial data $\mathbf{f}_0=[\mathbf{f}_0(k)\;:\; k \in \mathbb{Z}]\in\ell(\mathbb{Z})$ parametrized by the starting mesh $\mathbf{t}$ (see \eqref{eq:t}) to 
finer sequences $\mathbf{f}_j=[\mathbf{f}_j(k)\;:\; k \in \mathbb{Z}] \in \ell(\mathbb{Z})$ parametrized by the meshes $2^{-j}\mathbf{t}$, $j \in \mathbb{N}$ by
\begin{equation}\label{eq:sub_step}\mathbf{f}_j\;=\;\mathbf{P}^j\mathbf{f}_0,\quad j\in\mathbb{N}.
\end{equation}

In the regular case, i.e. $\mathbf{t}=\mathbb{Z}$, the subdivision matrix $\mathbf{P}$ is $2$-slanted  with 
\begin{equation}\label{eq:2slant}\mathbf{P}(k,m)\;=\;\mathbf{p}(k-2m),\quad \mathbf{p} \in \ell_0(\mathbf{Z}), \quad k,m\in\mathbb{Z}.
\end{equation}
We make use of the trigonometric polynomial (\emph{subdivision symbol})
\[p(\omega)\;=\;\frac{1}{2}\sum\limits_{k\in\mathbb{Z}}\;\mathbf{p}(k)\;e^{i2k\pi\omega}, \quad \omega \in [0,1),
\]
associated to the \emph{mask} $\mathbf{p} \in \ell_0(\mathbf{Z})$. The \emph{support of the mask $\mathbf{p}$} is the set of indeces
\[\supp(\mathbf{p})\;=\;\left\{\;\min\left\{k\in\mathbb{Z}\;:\;\mathbf{p}(k)\neq0\right\},\;\dots,\;\max\left\{k\in\mathbb{Z}\;:\;
\mathbf{p}(k)\neq0\right\}\;\right\}.\]

The subdivision scheme is said to be \emph{convergent} if, for every $k\in\mathbb{Z}$, there exists $\varphi_k\in C^0(\mathbb{R})$ called \emph{basic limit function} such that 
\begin{equation*}
\lim\limits_{j\rightarrow \infty} \sup_{m \in \mathbb{Z}} | \varphi_k(2^{-j m})- \mathbf{f}_j(m)|=\;0 \quad
\hbox{for} \quad \mathbf{f}_0(m)=\delta_k(m)= \left\{\begin{array}{cc}1, & m=k, \\ 0, & \hbox{otherwise}. \end{array} \right.
\end{equation*}
Equivalently, $\varphi_k$ must be the uniform limit of the piecewise linear functions that interpolate the data $\mathbf{f}_j$ on the mesh $2^{-j}\mathbf{t}$.
A well known necessary condition \cite{MR1079033} for convergence of subdivision states that $1$ is the right eigenvalue of 
$\mathbf{P}$ with algebraic multiplicity one and all other eigenvalues of $\mathbf{P}$ are in the absolute value
less than $1$. In terms of the subdivision symbol, this condition reads as $p(0)=1$.
The associated eigenvector is $\mathbf{1}=[1\;:\; k\in\mathbb{Z}]$. This implies the \emph{partition of unity} property
\begin{equation}\label{eq:POU} [\;\varphi_k\;]_{k\in\mathbb{Z}}^T\;\mathbf{1}\;\equiv\;1\;.\end{equation}
The basic limit functions satisfy the  \emph{refinement equation}
\begin{equation}\label{eq:ref_eq}
  [\varphi_k(x)\;:\;k\in\mathbb{Z}] = \mathbf{P}^T \; \left[\;\varphi_k(2x)\;:\; k\in\mathbb{Z}\right], \quad x \in \mathbb{R}.
\end{equation}
In the regular case, due to \eqref{eq:sub_step} and \eqref{eq:2slant}, we have $\varphi_k\;=\;\varphi_0(\cdot-k)$, $k\in\mathbb{Z}$
and the refinement equation becomes
\begin{equation*}
\widehat{\varphi}_0(\omega)\;=\;p(\omega/2)\;\widehat{\varphi}_0(\omega/2), \quad \omega \in [0,1),
\end{equation*}
with the Fourier transform $\widehat{\varphi}_0$ of $\varphi_0$ defined by
\[\widehat{\varphi}_0(\omega)\;=\;\int_{\mathbb{R}}\;e^{-i2\pi\omega x}\;\varphi_0(x)\;dx\;.\]
As for the semi-regular case, the subdivision matrix $\mathbf{P}$ can differ from the regular one at the columns whose indeces belong to the set
\begin{equation}\label{eq:I_irr_gen}
\mathcal{I}_{irr}\;=\;\left\{\;k\in\mathbb{Z}\;:\;-\max(\supp(\mathbf{p}))<k<-\min(\supp(\mathbf{p}))\;\right\}.
\end{equation}
This implies the loss of shift-invariace, i.e. the basic limit functions
$\{\varphi_k\}_{k\in \mathcal{I}_{irr}}$, which are the ones with $0$ in the interior of their support, are no longer
the shifts of any other basic limit function. Nevertheless, each of the functions $\varphi_k$, $k<\min(\mathcal{I}_{irr})$, and 
$\varphi_k$, $k>\max(\mathcal{I}_{irr})$, are basic limit functions of the corresponding regular subdivision schemes.

\noindent If the matrix 
\begin{equation}\label{eq:D_mat}
 \mathbf{D}\;=\;\diag\left(\;\int_{\mathbb{R}}\; [\varphi_k(x)\ : \ k\in \mathbb{Z}]\;dx\;\right)
\end{equation}
is positive definite we can define the corresponding scaling functions 
\begin{equation}\label{eq:renorm} 
   \Phi=[\phi_k\ : \ k\in \mathbb{Z}]=\mathbf{D}^{-1/2} \; [\varphi_k\ : \ k\in \mathbb{Z}]
\end{equation}
for which, by \eqref{eq:POU},
\[ \Phi(x)^T\;\int_{\mathbb{R}}\;\Phi(y)\;dy\;\equiv\;1\;,\]

\noindent The refinement equation \eqref{eq:ref_eq} holds also for $\Phi$ with a renormalized subdivision matrix 
\begin{equation}\label{eq:ref_eq_renorm}\Phi\;=\;\mathbf{D}^{-1/2}\;\mathbf{P}^T\;\mathbf{D}^{1/2}\;\Phi(2\;\cdot)\;.\end{equation}

\begin{defn} A \emph{semi-regular wavelet tight frame} of $L^2(\mathbb{R})$ is a family $\Phi\cup\{\Psi_j\;:\; j\in\mathbb{N}\}$ of 
$L^2(\mathbb{R})$ functions  such that
\begin{enumerate}
\item $\Phi$ generates a \emph{Multi Resolution Analysis} (MRA), i.e,
\begin{equation} \label{eq:MRA}
\begin{array}{c}
\overline{\lspan\{\phi_k(2^j\cdot)\}_{k\in\mathbb{Z}}}^{L^2}\;\subset\;\overline{\lspan\{\phi_k(2^{j+1}\cdot)\}_{k\in\mathbb{Z}}}^{L^2},\quad \forall j\in\mathbb{Z}, \\ \\
\{0\}\;=\;\bigcap\limits_{j\in\mathbb{Z}}\;\overline{\lspan\{\phi_k(2^j\cdot)\}_{k\in\mathbb{Z}}}^{L^2}, \\ \\
 L^2(\mathbb{R})\;=\;\overline{\bigcup\limits_{j\in\mathbb{Z}}\;\overline{\lspan\{\phi_k(2^j\cdot)\}_{k\in\mathbb{Z}}}^{L^2}}^{L^2},
\end{array}
\end{equation}
\item there exists a bi-infinite matrix $\mathbf{Q}$ such that\\
\begin{equation}\label{eq:framelets}\Psi_{j}\;=\;2^{j/2}\;\mathbf{Q}^T\;\Phi(2^j\cdot),\quad j\in\mathbb{N},\end{equation}
\item the \emph{tight frame} property (decomposition) holds, i.e.
\begin{equation*}
f\;=\;\sum\limits_{k\in\mathbb{Z}}\;\langle f,\phi_k\rangle \phi_k\;+\;\sum\limits_{j\in\mathbb{N}}\;\sum\limits_{k\in\mathbb{Z}}\;\langle f,\psi_{j,k}\rangle\;\psi_{j,k}\;,\quad  f\in L^2(\mathbb{R}).
\end{equation*}
\end{enumerate}
The functions $\phi_k$, $k\in\mathbb{Z}$, are called \emph{scaling functions} and $\psi_{j,k}$, $j\in\mathbb{N},k\in\mathbb{Z}$ are \emph{framelets}.
\end{defn}

\begin{defn} 
We say that a wavelet tight frame $\Phi\cup\{\Psi_j\;:\; j\in\mathbb{N}\}$ has $v\in\mathbb{N}$ vanishing moments if
\[\int_\mathbb{R}\;x^\alpha\;\Psi_j(x)\;dx\;=\;\mathbf{0},\quad \alpha\in\{0,\dots,v-1\}, \quad  j\in\mathbb{N}.\]
\end{defn}

\section{Dubuc-Deslauriers $2n$-point subdivision schemes and their scaling functions} \label{sec:DD}

\noindent The elements of the family of Dubuc-Deslauriers schemes, introduced in \cite{MR982724} in the regular case, are defined
in the semi-regular setting in \cite{MR1356994} as solutions of the following interpolation problems. Let $n\in\mathbb{N}$ and define the mesh $\mathbf{t}$ as in \eqref{eq:t}. We require that the subdivision matrix
$\mathbf{P}$ is interpolatory and maps any data $\pi(\mathbf{t})$, $\pi \in \Pi_{2n-1}$, into
$\pi(\mathbf{t}/2)=\mathbf{P} \pi(\mathbf{t})$ by using linear combinations of the $2n$ neighbouring points, i.e.
\begin{enumerate}
\item $\mathbf{P}(2k,k)=1$, $k\in\mathbb{Z}$,
\item the entries $\mathbf{P}_k=[\mathbf{P}(2k+1,j) \ : \ j=k-n+1, \ldots, k+n]$, $k \in \mathbb{Z}$,
 satisfy
$$ \mathbf{P}_k \ {\scriptsize\begin{bmatrix}
 1 & \mathbf{t}(k-n+1)  & \cdots& \mathbf{t}(k-n+1)^{2n-1}\\
 1 & \mathbf{t}(k-n+2)& \cdots & \mathbf{t}(k-n+2)^{2n-1}   \\
  \vdots& \vdots & \ddots & \vdots \\
  1  & \mathbf{t}(k+n) & \cdots & \mathbf{t}(k+n)^{2n-1}
\end{bmatrix}
=\;\begin{bmatrix}
 1 &
 \frac{\mathbf{t}(2k+1)}{2} &
 \cdots &
 \left(\frac{\mathbf{t}(2k+1)}{2}\right)^{2n-1}
\end{bmatrix}}.
$$
\item all other entries of $\mathbf{P}$ are equal to zero.
\end{enumerate}

\begin{prop}
Let $n\in\mathbb{N}$ and $\mathbf{P}$ be the subdivision matrix constructed in 1.-3. on the mesh $\mathbf{t}$. Then
\begin{enumerate}
\item[(i)] $1$ is a simple eigenvalue of $\mathbf{P}$ associated  to the right eigenvector $\mathbf{1}$ and all other
eigenvalues of $\mathbf{P}$ are less than $1$ in the absolute value;
\item[(ii)] the subdivision scheme with subdivision matrix $\mathbf{P}$ converges.
\end{enumerate}
\end{prop}

\begin{proof}
Part $(i)$: Let $n \in \mathbb{N}$ and $\mathcal{I}=\{1-2n,\dots,2n-1\}$. By \cite[Theorem 3]{MR1356994}, all non-zero eigenvalues of $\mathbf{P}$ are uniquely determined by the eigenvalues
of its finite section, the square matrix $\widetilde{\mathbf{P}}=(\mathbf{P}(m,k))_{m,k \in \mathcal{I}}$.  By
construction, due to the interpolation property, $\widetilde{\mathbf{P}}\mathbf{1}=\mathbf{1}$. We show next, that $\lambda \in \mathbb{C} 
\setminus\{0,1\}$ with $\widetilde{\mathbf{P}}\mathbf{v}=\lambda \mathbf{v}$, $\mathbf{v} \in \mathbb{C}^{4n-3} \setminus\{\mathbf{0}\}$, must satisfy $|\lambda|<1$.
The proof is by contradiction, we assume that $|\lambda| \ge 1$. Note first that $\widetilde{\mathbf{P}}(0,0)=1$ and by step 3. of the construction above, we get $\mathbf{v}(0)=\lambda\; \mathbf{v}(0)$, thus, $\mathbf{v}(0)=0$. Step 1. of the construction
forces $\mathbf{v}(k)=\lambda \ \mathbf{v}(2k)$ for $k,2k \in \mathcal{I}$. To determine the odd entries of $\mathbf{v}$, let $m \in \mathcal{I}$ be odd and consider the polynomial interpolation problem with the pairwise distinct knots
$\mathbf{t}(j)$ and values $\mathbf{v}(j)$ for $j \in \left\{\frac{m+1}{2}-n, \ldots, n+\frac{m+1}{2} \right\}$.
This interpolation problem possesses a unique solution, possibly complex-valued, interpolation polynomial $\pi \in \Pi_{2n-1}$. Therefore, by the interpolation property of $\mathbf{P}$, we have
$$
 \lambda \ \pi \left(\mathbf{t}(m) \right) \;=\;\lambda \ \mathbf{v}(m)\;=\;\left(\widetilde{\mathbf{P}} \ \mathbf{v}\right)(m) \;=\; \pi \left(\frac{\mathbf{t}(m)}{2} \right).
$$ 
Iterating we obtain
\begin{equation} \label{eq:aux}
 \lim_{r \rightarrow \infty} |\lambda|^r \ |\pi\left( \mathbf{t}(m)\right)|\;=\;\lim_{r \rightarrow \infty} \left| \pi \left(\frac{\mathbf{t}(m)}{2^r} \right) \right|\;=\;|\pi(0)|\;=\;|\mathbf{v}(0)|\;=\;\mathbf{0},
\end{equation}
which leads to a contradiction: for $|\lambda|=1$, we have $\mathbf{v}(m)=\pi(\mathbf{t}(m))=\mathbf{0}$ or, for $|\lambda|>1$,  the identity \eqref{eq:aux} is violated. It is left to show that $\lambda=1$ is simple. The proof is by contradiction. W.l.o.g. we assume that $1$ has an algebraic multiplicity $2$. Note that $\mathbf{\delta}^T \widetilde{\mathbf{P}}=
\mathbf{\delta}^T$, where $\mathbf{\delta}(0)=1$ and its other entries are equal to zero and define  
$
 \mathbf{B}=\widetilde{\mathbf{P}}-\mathbf{1} \mathbf{\delta}^T$. Then $\mathbf{B}$ has a simple eigenvalue $1$ with $\mathbf{B}\;\mathbf{v}=\mathbf{v}$, $\mathbf{v} \not=\mathbf{0}$. By construction $\mathbf{v}(0)=0$, thus $\mathbf{B}\mathbf{v}=\mathbf{\widetilde{P}}\mathbf{v}$. Following
a similar argument as above we arrive at the contradiction $\mathbf{v}=\mathbf{0}$.

\noindent Part $(ii)$: To prove the convergence of the scheme, due to \cite{MR982724,MR1356994}, it suffices to prove the continuity of the basic limit functions in $0$, i.e. 
\begin{equation}\label{eq:aux1}
|\mathbf{f}_{j,k}(0)-\mathbf{f}_{j,k}(1)| \;\underset{j\rightarrow\infty}{\longrightarrow}\;0, \quad \hbox{and} \quad |\mathbf{f}_{j,k}(0)-\mathbf{f}_{j,k}(-1)|\;\underset{j\rightarrow\infty}{\longrightarrow}\;0
\end{equation}
for $\mathbf{f}_{j,k}\;=\;\mathbf{P}^j\;\delta_{k}$, $k\in\mathbb{Z}$. Or, equivalently, we show that
\[
 |\mathbf{f}_{j,k}(0)\;-\;\mathbf{f}_{j,k}(1)|\;=\;|[\;0,\dots,0,\underbrace{1}_{\hbox{$0$-th}},-1,0,\dots,0\;]\;\mathbf{\widetilde{P}}^j\;\delta_{k}| \rightarrow 0
\]
and, similarly, for the other difference in \eqref{eq:aux1}. The claim follows then by part $(i)$, together with steps 1. and 3. of the construction, which imply that $\mathbf{\widetilde{P}}^j\;\delta_{k}\rightarrow\delta_k(0)\;\mathbf{1}$.
\end{proof}

The structure of $\mathbf{P}$ implies that $\supp([\mathbf{P}(m,k)\ :\ m\in\mathbb{Z}])=\{2k-2n+1,\dots,2k+2n-1\}$
and, thus, by \cite{MR2775138}, the corresponding basic limit functions satisfy
\begin{equation}\label{eq:phi_supp}
 \supp(\varphi_k)=[\mathbf{t}(k+1-2n),\mathbf{t}(k+2n-1)], \quad k \in \mathbb{Z}.
\end{equation}
The construction of $\mathbf{P}$ in $1.-3.$ does not only unify regular and the semi-regular cases, but also ensures that
the corresponding subdivision schemes are interpolatory and reproduce polynomials up to degree $2n-1$. In particular,
the convergence and the interpolation property imply that the functions $\{ \varphi_k \ : \  k \in\mathbb{Z}\}$ are linearly independent 
and, thus, the representation
\begin{equation}\label{eq:DD_rep}
  x^\alpha\;=\;\sum\limits_{k\in\mathbb{Z}}\;\mathbf{t}^\alpha(k)\;\varphi_k(x)\;,\quad  x\in\mathbb{R},
	\quad \alpha\in\{0,\dots,2n-1\},
\end{equation}
is unique. The existence of the corresponding scaling functions $\phi_k$ in \eqref{eq:renorm}  depends on the choice of $h_\ell$ and $h_r$ in \eqref{eq:t},
which ensures that the integrals of the basic limit functions appearing in \eqref{eq:D_mat} are positive,
see Section \ref{sec:examples}. Under this assumption, the corresponding scaling functions $\phi_k$ in the column vector $\Phi=[\phi_k\ : \ k\in \mathbb{Z}]$ in \eqref{eq:renorm} inherit the polynomial reproduction property in \eqref{eq:DD_rep} and we have
\begin{equation}\label{eq:DD_rep_renorm}
 x^\alpha\;=\;\Phi^T(x) \;\mathbf{c}_\alpha,\quad \mathbf{c}_\alpha=\mathbf{D}^{1/2}\;\mathbf{t}^\alpha,
 \quad  \alpha\in\{0,\dots,2n-1\} .
\end{equation}
Furthermore, note that, in the semi-regular case,
\eqref{eq:I_irr_gen} becomes
\begin{equation}\label{eq:I_irr}
 \mathcal{I}_{irr}\;=\;\{2-2n,\dots,2n-2\}\quad\textrm{with}\quad|\mathcal{I}_{irr}|\;=\;4n-3,
\end{equation}
and, denoted with $\mathbf{\Phi}_\ell$ and $\mathbf{\Phi}_r$ the vectors of scaling functions on the regular meshes $\mathbf{t}_\ell=h_\ell\mathbb{Z}$ and $\mathbf{t}_r=h_r\mathbb{Z}$ respectively, we have
\begin{equation} \label{def:Phi_l_Phi_r}
\mathbf{I}_\ell\;\Phi\;=\;\mathbf{I}_\ell\;\Phi_\ell\quad\textrm{ and }\quad \mathbf{I}_r\;\Phi\;=\;\mathbf{I}_r\;\Phi_r,
\end{equation}
where
\[\mathbf{I}_\ell(j,k)\;=\;\left\{\begin{array}{cl}
1,&\textrm{if } j=k<2-2n, \\ \\
0,&\textrm{otherwise},
\end{array}\right. 
\quad\textrm{ and }\quad
\mathbf{I}_r(j,k)\;=\;\left\{\begin{array}{cl}
1,& \textrm{if }j=k>2n-2 \\ \\
0,&\textrm{otherwise}.
\end{array}\right. \]

\section{Regular Case : Unitary Extension Principle} \label{sec:UEP}

\noindent In the regular case, our wavelet tight frame construction, see Proposition \ref{cor:regDDframe}, relies on the well known \emph{Unitary Extension Principle} \cite{MR1469348}. Note that the interpolation property of Dubuc-Deslauriers $2n$-point schemes guarantees
$n$ vanishing moments for the corresponding framelets, higher than it is usually expected for a general UEP based construction. 
Thus, in the regular setting, we avoid using the more general and demanding \emph{Oblique Extension Principle} \cite{MR1968118, MR1971300}.
However, in the next section, the Oblique Extension Principle appears, see Remark \ref{rem:time_domain_Chui}, as a special case of the 
irregular framework presented in \cite{MR2110512}.

\begin{thm}[Unitary Extension Principle (UEP), \cite{MR1469348}] \label{thm:UEP}
Let $p(\omega)$ be the symbol of a convergent subdivision scheme. If the trigonometric polynomials
$q_j(\omega)=\frac{1}{2} \displaystyle \sum_{k\in\mathbb{Z}} \mathbf{q}_j(k) \;e^{i2 k \pi\omega}$, $j=1, \ldots, J$,
satisfy
\begin{equation}\label{eq:UEP}
\left\{\begin{array}{rl}
p(\omega)\overline{p(\omega)} \;+\; \sum\limits_{j=1}^{J}\;q_j(\omega)\overline{q_j(\omega)}\;=\;1 \\
p(\omega)\overline{p(\omega-1/2)} \;+\; \sum\limits_{j=1}^{J}\;q_j(\omega)\overline{q_j(\omega-1/2)}\;=\;0 \\
\end{array}\right.,\quad\omega\in[0,1),
\end{equation}
then the bi-infinite matrix $\mathbf{Q}$ with entries $\mathbf{Q}(k,m+j-1)=2^{-1/2} \mathbf{q}_j(k-2m)$,
$k,m\in\mathbb{Z}$, $j=1, \ldots, J$,
defines a wavelet tight frame  in \eqref{eq:framelets}.
\end{thm}

\begin{rem} We call the trigonometric polynomials in Theorem \ref{thm:UEP} \emph{framelet symbols}.
The number $v \in\mathbb{N}$ of the vanishing moments of the corresponding wavelet tight frame is given by (see e.g. \cite{MR1971300})
\[v=\min\limits_{j\in J}\left\{\; \mu_j \in \mathbb{N} \ : \  q^{(k)}_j(0)=0, \ k=0,\ldots, \mu_j-1 \ \hbox{and} \  q^{(\mu_j)}_j(0)\not =0  \;\right\}.\]
\end{rem}

\noindent Theorem \ref{thm:UEP} holds, in particular, for orthonormal Daubechies wavelets. Indeed, the Daubechies $2n$-tap scheme \cite{MR1162107} has the scaling symbol $d(\omega)$ with coefficients $\mathbf{d}$ supported on $\{1-2n,\dots,0\}$ which satisfies \eqref{eq:UEP} with
\begin{equation}\label{eq:Daub_wav}
 q_d(\omega)=q_1(\omega)\;=\;e^{-i2\pi (2n-1) \omega}\;\overline{d(\omega-1/2)} \quad \hbox{and} \quad v=n, \quad n\in\mathbb{N}.
\end{equation}
Daubechies wavelets are closely connected to the Dubuc-Deslauriers subdivision schemes, see \cite{MR1397613}. Indeed, the
symbol $d(\omega)$ of the Daubechies $2n$-tap scheme satisfies
\begin{equation}\label{eq:Micchelli}
 p(\omega)\;=\;d(\omega)\;\overline{d(\omega)}, \quad \omega \in [0,1),
\end{equation}
where $p(\omega)$ is the symbol of the Dubuc-Deslauriers $2n$-point scheme.
The identity \eqref{eq:Micchelli}, together with Theorem \ref{thm:UEP_mult}, lead to our UEP construction
of Dubuc-Deslauriers wavelet tight frame in Proposition \ref{cor:regDDframe}.

\begin{prop} \label{cor:regDDframe}
Let $n\in\mathbb{N}$, $d$ and $p$ be the symbols of the Daubechies $2n$-tap scheme and the Dubuc-Deslauriers $2n$-point scheme, respectively. Then
\[ q_1(\omega)\;=\;\sqrt{2}\;e^{i2\pi(2n-1)\omega}d(\omega)\;d(\omega-1/2) \quad\textrm{and}\quad q_2(\omega)\;=\;d(\omega-1/2)\;\overline{d(\omega-1/2)}, \quad \omega \in [0,1), \]
define a wavelet tight frame (with $n$ vanishing moments) for $p$.
\end{prop}

\begin{proof}
Note that the convergence of the subdivision associated to $d(\omega)$ implies the convergence of the subdivision
associated to $\overline{d(\omega)}$. Moreover, by conjugating the second identity in \eqref{eq:UEP}, we see that the UEP identities in \eqref{eq:UEP} are also satisfied by $\overline{d(\omega)}$. Thus, applying Theorem \ref{thm:UEP_mult}
to $d$ and $\overline{d}$, multiplying the resulting identities and using \eqref{eq:Micchelli}, we obtain a wavelet tight frame for $p=d\overline{d}$ with $n$ vanishing moments and consisting of framelets
\[
 d(\omega)\;\overline{q_d(\omega)}, \quad \overline{d(\omega)}\;q_d(\omega),\quad\textrm{and}\quad q_d(\omega)\;\overline{q_d(\omega)}.
\]
To reduce the number of frame generators we take a closer look to the structure of these framelets. Indeed, the UEP identities  \eqref{eq:UEP} and \eqref{eq:Daub_wav} yield
\small
\begin{equation}\label{eq:DDUEP} \begin{array}{rcl}
1&=&\;\left(\;d(\omega)\;\overline{d(\omega)}\;+\;q_d(\omega)\;\overline{q_d(\omega)}\;\right)\;\left(\;\overline{d(\omega)}\;d(\omega)\;+\;\overline{q_d(\omega)}\;q_d(\omega)\;\right) \\ \\
&=& \left(d(\omega)\;\overline{d(\omega)}\right)\;\overline{\left(d(\omega)\;\overline{d(\omega)}\right)}\;+\;\sqrt{2}\;d(\omega)\;\overline{q_d(\omega)}\;\overline{\left(\sqrt{2}\;d(\omega)\;\overline{q_d(\omega)}\right)}\;+\;\left(q_d(\omega)\;\overline{q_d(\omega)}\right)\;\overline{\left(q_d(\omega)\;\overline{q_d(\omega)}\right)},\\ \\ \\
0&=&\;\left(\;d(\omega)\;\overline{d(\omega-1/2)}\;+\;q_d(\omega)\;\overline{q_d(\omega-1/2)}\;\right)\;\left(\;\overline{d(\omega)}\;d(\omega-1/2)\;+\;\overline{q_d(\omega)}\;q_d(\omega-1/2)\;\right) \\ \\
&=& \left(d(\omega)\;\overline{d(\omega)}\right)\;\overline{\left(d(\omega-1/2)\;\overline{d(\omega-1/2)}\right)}\;+\;\left(q_d(\omega)\;\overline{q_d(\omega)}\right)\;\overline{\left(q_d(\omega-1/2)\;\overline{q_d(\omega-1/2)}\right)} \\ \\
& & \qquad+\;d(\omega)\;\overline{q_d(\omega)}\;\overline{\left(d(\omega-1/2)\;\overline{q_d(\omega-1/2)}\right)}\;+\;\overline{d(\omega)}\;q_d(\omega)\;d(\omega-1/2)\;\overline{q_d(\omega-1/2)}, \quad \omega \in [0,1).
\end{array} \end{equation}
\normalsize
From \eqref{eq:Daub_wav}, we have
\[
 q_d(\omega)\;d(\omega-1/2)\;=\;\overline{q_d(\omega)}\;\overline{d(\omega-1/2)}, \quad \omega \in [0,1).
\]
Moreover, the periodicity of the symbols implies
\[
 \overline{d(\omega)}\;\overline{q_d(\omega-1/2)}\;=\;d(\omega)\;q_d(\omega-1/2), \quad \omega \in [0,1).
\]
Next, we rewrite the last term of the second identity in \eqref{eq:DDUEP}
\[\overline{d(\omega)}\;q_d(\omega)\;d(\omega-1/2)\;\overline{q_d(\omega-1/2)}\;=\;d(\omega)\;\overline{q_d(\omega)}\;\overline{\left(d(\omega-1/2)\;\overline{q_d(\omega-1/2)}\right)}, \quad \omega \in [0,1),
\]
obtaining the framelets $q_1(\omega)=\sqrt{2}d(\omega)\overline{q_d(\omega)}$ and
$q_2(\omega)=q_d(\omega)\overline{q_d(\omega)}$. The claim follows by \eqref{eq:Daub_wav}.
\end{proof}

\noindent This result is equivalent to the one in \cite[Section 3.1.2]{MR1754930}. The proof presented here, however, 
can be easily adapted to get the following more general result.

\begin{thm}\label{thm:UEP_mult}
Let $a$ and $p$ be trigonometric polynomials that satisfy \eqref{eq:UEP} with
$b_j$, $j=1, \ldots, J_a$ and $q_j$, $j=1, \ldots, J_p$, respectively.
Then the product $ap$ satisfies \eqref{eq:UEP} with
\[
 \big\{ p\, b_j:  j=1, \ldots,J_a \big\}\  \cup \  \{ a\, q_j \ : \ j=1, \ldots, J_p\}
\;\cup\;\big\{q_j \, b_k \  :\ j=1,\ldots,J_p, \ k=1,\ldots,J_a \}.
\]
Moreover, if the schemes associated to $a$, $p$ and $ap$ are convergent and the corresponding framelets $b_j$, $j=1,\ldots,J_a$ and $q_j$, $j=1, \ldots,J_p$ have $v_a$ and $v_p$ vanishing moments, respectively, then the wavelet tight frame for the
trigonometric polynomial $a p$
has $v=\min\{v_a,v_p\}$ vanishing moments.
\end{thm}

\begin{rem}
Theorem \ref{thm:UEP_mult} can be used to obtain explicit algebraic expressions for framelets with $1$ vanishing moment
for the family of B-spline schemes, since the B-spline symbols are products of Haar symbols. The same trick as in the proof
of Proposition~\ref{cor:regDDframe}, allows us to reduce the number of framelets to the order of the corresponding B-spline.
This is an alternative, more straightforward way for obtaining the framelets in \cite{MR2419706} in the regular case.
\end{rem}

\section{Semi-regular Case : Matrix Unitary Extension Principle} \label{sec:semi_reg}

\noindent Let $n\in\mathbb{N}$ and $h_\ell,h_r>0$. In this section,  we construct a semi-regular wavelet tight frame from the column vector 
$\Phi=[\phi_k : k\in \mathbb{Z}]$
of scaling functions $\phi_k$ generated by the semi-regular Dubuc-Deslauriers $2n$-point scheme on the mesh $\mathbf{t}$ in \eqref{eq:t}. 
This semi-regular wavelet tight frame has $n$ vanishing moments and its construction is based on the time domain characterization in
\cite{MR2110512}. To be able to apply the result in \cite{MR2110512}, we check  first certain properties of $\Phi$, see subsection
\ref{subsec:properties_DD}. The actual construction of the corresponding wavelet tight frame and its properties are presented in 
subsection \ref{subsec:DDWTF}.

\subsection{Properties of Dubuc-Deslauriers scaling functions $\Phi$} \label{subsec:properties_DD}

\begin{lem}\label{lem:Lemma1}
\begin{enumerate}
\item[$(i)$] $\Phi$ is a \emph{Riesz basis} for $\overline{\lspan(\Phi)}^{L^2}$, i.e. there exist $0<A\leq B< \infty$ such that
\[A\;\|\mathbf{f}\|^2_{\ell^2}\;\leq\;\|\Phi^T\mathbf{f}\;\|^2_{L^2}\;\leq\;B\;\|\mathbf{f}\|^2_{\ell^2},\quad \mathbf{f}\in\ell^2(\mathbb{Z}).\]
\item[$(ii)$] $\Phi$ is \emph{uniformly bounded}, i.e. $\sup\limits_{k\in\mathbb{Z}}\|\phi_k\|_{L^\infty}< \infty$.
\item[$(iii)$] $\Phi$ is \emph{strictly local}, i.e. $\sup\limits_{k\in\mathbb{Z}}|\supp(\phi_k)|< \infty$ and there exists $s \in \mathbb{N}$
 such that, for every index set $\mathcal{I}\subset\mathbb{Z}$ with $|\mathcal{I}|>s$,
\[\bigcap\limits_{k\in\mathbb{\mathcal{I}}}\;\supp(\phi_k)\;=\;\emptyset.\]
\item[$(iv)$] $\lim\limits_{j\rightarrow \infty}\sup\limits_{k\in\mathbb{Z}}|\supp \phi_k(2^j\;\cdot)|=0$. 
\end{enumerate}
\end{lem}

\begin{proof}
$(i)$: Let $B_\ell >0$ and $ B_r >0$ be the Riesz constants for the regular scaling functions $\Phi_\ell$ and 
$\Phi_r$ \eqref{def:Phi_l_Phi_r}, respectively. Then, for every $\mathbf{f}\in\ell_2(\mathbb{Z})$,
$$
\|\Phi^T\mathbf{f}\|^2_{L^2} =  \left\|\;\left(\sum\limits_{k<\min\{\mathcal{I}_{irr}\}}\;+\;\sum\limits_{k\in\mathcal{I}_{irr}}\;+\;\sum\limits_{k>\max\{\mathcal{I}_{irr}\}}\right)\;\mathbf{f}(k)\phi_k\;\right\|^2_{L^2} 
 \leq  \left(B_\ell\;+\;\max\limits_{k\in\mathcal{I}_{irr}}\|\phi_k\|_{L^2}^2\;+\;B_r\right)\;\|\mathbf{f}\|^2_{\ell^2}\;.
$$
The existence of $A>0$ follows from the linear independence of functions in $\Phi$. \\
Part $(ii)$ follows directly, due to the continuity and compact support of the functions in $\Phi$.\\
$(iii)$: For every $k\in\mathbb{Z}$, by \eqref{eq:phi_supp} and by the definition of $\mathbf{t}$, we have
\[\sup\limits_{k\in\mathbb{Z}}|\supp(\phi_k)|\;=\;\max\{h_\ell,h_r\}(4n-2)\;<\;\infty.\] 
Thus, choose $s=4n-2$. \\ 
The claim $(iv)$ follows directly from $(iii)$.
\end{proof}

\noindent Lemma \ref{lem:Lemma2} shows an important property of the bi-infinite \emph{Gramian matrix} (of $\Phi$) 
\begin{equation*}
\mathbf{G}_{\Phi}\;=\;\int_{\mathbb{R}}\;\Phi(x)\;\Phi(x)^T\;dx\;.
\end{equation*}

\begin{lem}\label{lem:Lemma2} For $\mathbf{c}_\alpha= \mathbf{D}^{1/2} \mathbf{t}^{\alpha}$, $\alpha\in\{0,\dots,2n-1\}$,
we have
$$
\mathbf{G}_{\Phi}^{-1}\;\mathbf{m}_\alpha\;=\;\mathbf{c}_\alpha, \quad \alpha\in\{0,\dots,2n-1\},
$$
where $\mathbf{m}_\alpha$ are the \emph{moments} of $\Phi$, i.e.
\begin{equation}\label{eq:moments}
     \mathbf{m}_\alpha\;=\;\int_{\mathbb{R}}\;x^\alpha\;\Phi(x)\;dx.
\end{equation}
\end{lem}
\begin{proof}
By \cite{MR2110512}, part $(i)$ of Lemma~\ref{lem:Lemma1} guarantees that $\mathbf{G}_{\Phi}$ is a bounded 
invertible operator on $\ell^2(\mathbb{Z})$ with the bounded inverse.
By \eqref{eq:DD_rep_renorm}, we get
\begin{equation*}
\mathbf{G}_{\Phi}\;\mathbf{c}_\alpha\;=\;\int_\mathbb{R}\;\Phi(x)\;\Phi(x)^T\;\mathbf{c}_\alpha\;dx\;=\;\int_{\mathbb{R}}\;x^\alpha\;\Phi(x)\;dx\;=\;\mathbf{m}_\alpha, \quad \alpha\in\{0,\dots,2n-1\}.
\end{equation*}
\end{proof}

\noindent For computation of the moments in section~\ref{sec:examples}, by \cite{MR2692268}, we first determine $\mathbf{G}_{\Phi}$
by solving the system in \eqref{eq:mat_eq}. 
Theorem \ref{thm:semireg_prod_mat} gives a formal justification for the algorithm presented in \cite{MR2692268} for B-splines. 

\begin{thm} \label{thm:semireg_prod_mat}
Let
\[\mathbf{G}\;=\;\int_{\mathbb{R}}\;[\varphi_k(x)\ :\ k\in\mathbb{Z}]\;[\varphi_k(x)\ :\ k\in\mathbb{Z}]^T\;dx.\]
Assume that $[\mathbf{G}(k,m)]_{k,m\not\in\mathcal{I}_{irr}}$ are given.
The rest of the elements of  $\mathbf{G}$ are uniquely determined from the linear system 
of equations given by
\begin{equation} \label{eq:mat_eq}
\mathbf{G}\;=\;\frac{1}{2}\;\mathbf{P}^T\;\mathbf{G}\;\mathbf{P}.
\end{equation}
Moreover 
\begin{equation} \label{eq:G_rec}
\mathbf{G}_\Phi\;=\;\mathbf{D}^{-1/2}\;\mathbf{G}\;\mathbf{D}^{-1/2},\quad\textrm{where}\quad\mathbf{D}\;=\;\diag(\mathbf{G}\cdot\mathbf{1}).
\end{equation}
\end{thm}

\begin{proof}
Observe first that $\mathbf{G}$ satisfies \eqref{eq:mat_eq}. Indeed, using the refinement equation \eqref{eq:ref_eq} and the substitution $y=2x$, we have
\[\begin{array}{rcl}
 \mathbf{G}&=&\int_\mathbb{R}\;[\varphi_k(x)\ :\ k\in\mathbb{Z}]\;[\varphi_k(x)\ :\ k\in\mathbb{Z}]^T\;dx\\ \\
&=&\mathbf{P}^T\;\int_\mathbb{R}\;[\varphi_k(2x)\ :\ k\in\mathbb{Z}]\;[\varphi_k(2x)\ :\ k\in\mathbb{Z}]^T\;dx\;\mathbf{P}\;=\;\frac{1}{2}\;\mathbf{P}^T\;\mathbf{G}\;\mathbf{P}.
\end{array}\]
Suppose next that there exist another solution to \eqref{eq:mat_eq}, i.e. another bi-infinite matrix  $\widetilde{\mathbf{G}}$. 
Then $\mathbf{H}\;= \mathbf{G}- \widetilde{\mathbf{G}}$ is non-zero and, by definition of ${\cal I}_{irr}$, has all its non-zero entries in the block
$ (\mathbf{H}_{i,j})_{5-6n \le i,j \le 6n-5}$. By linearity, $\mathbf{H}$ also satisfies \eqref{eq:mat_eq}. Define
the index sets
\begin{eqnarray*}
 \mathcal{I}_{\mathbf{H}}=\{\;k\in\mathbb{Z}\;:\;\exists m\in\mathbb{Z}\;\textrm{ with }\;\mathbf{H}(m,k)\neq0 \; \textrm{ or } \; \mathbf{H}(k,m)\neq0\;\}, \quad
 \mathcal{I}_{\mathbf{P}}=\{\;k\in\mathbb{Z}\;:\;\mathbf{P}(k,k)\neq 0\;\},
\end{eqnarray*}
and the set $\mathcal{I}$ as the intersection of $\mathbb{Z}$ with the convex hull of 
$\mathcal{I}_{\mathbf{H}}\cup\mathcal{I}_{\mathbf{P}}$. Furthermore, define the square matrices
$
 \widetilde{\mathbf{H}}=\left[\mathbf{H}(i,j)\right]_{i,j \in \mathcal{I}} \quad
 \hbox{and} \quad \widetilde{\mathbf{P}}=\left[\mathbf{P}(i,j)\right]_{i,j \in \mathcal{I}}$.
We obtain an equivalent finite version of \eqref{eq:mat_eq}
\begin{equation}\label{eq:mat_eq_cut} 
 \widetilde{\mathbf{H}}\;=\;\frac{1}{2}\;\widetilde{\mathbf{P}}^T\;\widetilde{\mathbf{H}}\;\widetilde{\mathbf{P}}.
\end{equation}
Due to  $\widetilde{\mathbf{H}}\not\equiv\mathbf{0}$, there exist $k,m \in \mathbb{N}$ such that $\widetilde{\mathbf{H}}(k,m) \neq 0$. Consider the canonical unit vectors $\mathbf{e}_k$ and $\mathbf{e}_m$ of $\mathbb{R}^{|\mathcal{I}|}$. Iterating \eqref{eq:mat_eq_cut}, we have 
\[0\;\neq\;\widetilde{\mathbf{H}}(k,m) \;=\;\mathbf{e}_k^T\;\widetilde{\mathbf{H}}\;\mathbf{e}_m\;=\;\frac{1}{2^j}\;\mathbf{e}_k^T\;(\mathbf{\widetilde{P}}^T)^j\;\widetilde{\mathbf{H}}\;(\mathbf{\widetilde{P}})^j\;\mathbf{e}_m, \quad j \in \mathbb{N}.
\]
Since $\widetilde{\mathbf{P}}$ is constructed to contain all the corresponding non-zero elements in the rows of $\mathbf{P}$, which is possible due to its 2-slanted nature, these two matrices have the same right eigenvalues. Since $\mathbf{P}$ has largest eigenvalue $1$ in absolute value with multiplicity one (see e.g. \cite{MR1356994}), $\|\mathbf{e}_k^T\;(\widetilde{\mathbf{P}}^T)^j\|_2  \leq C < \infty$ for all $j\in\mathbb{N}$ and, therefore,
\[0\;<\;|\;\mathbf{e}_k^T\;\widetilde{\mathbf{H}}\;\mathbf{e}_m\;|\;\leq\;\left|\;\frac{1}{2^j}\;\mathbf{e}_k^T\;(\widetilde{\mathbf{P}}^T)^j\;\widetilde{\mathbf{H}}\;(\widetilde{\mathbf{P}})^j\;\mathbf{e}_m\;\right|\;\leq\;\frac{1}{2^j}\;C^2\; \|\widetilde{\mathbf{H}}\|_2 \;\underset{j \rightarrow +\infty}{\longrightarrow}\;0,\]
which leads to a contradiction for the non-uniqueness of $\mathbf{G}$. Identity \eqref{eq:G_rec} follows directly from \eqref{eq:renorm} and \eqref{eq:POU}.
\end{proof}

\begin{rem}
$(i)$ The assumption about $[\mathbf{G}(k,m)]_{k,m\not\in\mathcal{I}_{irr}}$ is not restrictive. For the computation of such scalar products one can use \cite{MR1211402,MR1302255}, with the correct necessary renormalization on $h_\ell\mathbb{Z}$ and $h_r\mathbb{Z}$. \\
$(ii)$ Thanks to Theorem \ref{thm:semireg_prod_mat} and Lemma \ref{lem:Lemma2} we are able to compute all the vectors $\mathbf{c}_\alpha$ and $\mathbf{m}_\alpha$, for every $\alpha\in\{0,\dots,2n-1\}$.\\
$(iii)$ The result of Theorem \ref{thm:semireg_prod_mat} can be generalized to the case of two different semi-regular subdivision with subdivision matrices $\mathbf{A}$, $\mathbf{B}$ and basic limit functions $[\alpha_k \ :\ k\in\mathbb{Z}]$, $[\beta_k \ :\ k\in\mathbb{Z}]$ respectively, to compute $\int_\mathbb{R}[\alpha_k(x) \ :\ k\in\mathbb{Z}] [\beta_k(x) \ :\ k\in\mathbb{Z}]^Tdx$.
This is a crucial tool for applications. 
\end{rem}

\subsection{Semi-regular Dubuc-Deslauriers wavelet tight frame} \label{subsec:DDWTF}

\noindent We start our construction of semi-regular Dubuc-Deslauriers ($2n$-point) wavelet tight frames by building a bi-infinite matrix $\mathbf{S}$, see Algorithm~1. In Proposition~\ref{prop:properties_S} and Conjecture~\ref{conjecture},  we list and analyze the properties of 
$\mathbf{S}$ that, by \cite[Theorem 4.2]{MR2110512}, guarantee the wavelet tight frame property of the families in \eqref{def:Psi}.
Proposition~\ref{prop:eigenvector} ensures that the wavelet tight frame in \eqref{def:Psi} has $n$ vanishing moments.

\vspace{0.2cm}{\bf Algorithm 1:}
\begin{enumerate}
\item Define the $(4n-3)\times n$ matrix $\mathbf{C}\;=\;\begin{bmatrix} [\mathbf{c}_0(k)]_{k\in\{2-2n,\dots,2n-2\}}&|&\dots&|&[\mathbf{c}_{n-1}(k)]_{k\in\{2-2n,\dots,2n-2\}}\end{bmatrix}$,
where the bi-infinite column vectors $\mathbf{c}_\alpha$, $\alpha=0, \ldots, n-1$, are defined in \eqref{eq:DD_rep_renorm};
\item Compute the QR factorization $\mathbf{C}=\mathbf{O}\mathbf{U}$
with orthogonal $\mathbf{O}\in\mathbb{R}^{(4n-3)\times (4n-3)}$ and upper triangular $\mathbf{U}\in\mathbb{R}^{(4n-3)\times n}$;
\item Define $\mathbf{S}=\mathbf{I}$ and, if $h_\ell\neq h_r$, modify
\begin{equation}\label{eq:our_S} 
  \mathbf{S}_{irr}\;:=\;[\mathbf{S}(k,m)]_{k,m\in\mathcal{I}_{irr}}\;=\mathbf{\widetilde{O}}\;\mathbf{\widetilde{O}}^T, \qquad \mathbf{\widetilde{O}}=\left[\;[\mathbf{O}(k,1)]_{k\in\mathcal{I}_{irr}}\;|\;\dots\;|\;[\mathbf{O}(k,n)]_{k\in\mathcal{I}_{irr}}\;\right].
\end{equation}
\end{enumerate}

\noindent The matrix $\mathbf{S}$ from Algorithm 1 defines a kernel $\Phi(x)^T\mathbf{S}\Phi(y)$ with desired approximation properties.

\begin{prop} \label{prop:properties_S}  Let  $\mathbf{S}$ be defined by Algorithm~1. Then, for all $f\in L^2(\mathbb{R})$,
\begin{enumerate}
\item[$(i)$] $ \displaystyle \exists C>0\;:\; \int_{\mathbb{R}^2}\;f(x)\;\Phi(x)^T\;\mathbf{S}\;\Phi(y)\;f(y)\;dx\;dy\;\leq\;C\;\|f\|^2_{L^2}$,
\item[$(ii)$] $\displaystyle 2^j\;\int_{\mathbb{R}^2}\;f(x)\;\Phi(2^jx)^T\;\mathbf{S}\;\Phi(2^jy)\;f(y)\;dx\;dy\;\longrightarrow\;\left\{\begin{array}{rl} 0, & \textrm{if} \; j\rightarrow-\infty, \\
\|f\|^2_{L^2}, & \textrm{if} \;  j\rightarrow+\infty.
\end{array}\right. $
\end{enumerate}
\end{prop}

\begin{proof} Part $(i)$: Recall that the regular elements in $\Phi$ are the elements of $\Phi_\ell$ and $\Phi_r$ in \eqref{def:Phi_l_Phi_r}. 
For such regular families 
of scaling functions the characterization in \cite[Theorem 4.2]{MR2110512} with $\mathbf{S}=\mathbf{I}$ implies the existence
of  $C_\ell>0$ and $C_r>0$, respectively, such that, for all $f\in L^2(\mathbb{R})$,
\[\max\left\{\;\int_{\mathbb{R}^2}\;f(x)\;\Phi_\ell(x)^T\;\Phi_\ell(y)\;f(y)\;dx\;dy,\;\int_{\mathbb{R}^2}\;f(x)\;\Phi_r(x)^T\;\Phi_r(y)\;f(y)\;dx\;dy\;\right\}\;\leq\;\max\{C_\ell,C_r\}\;\|f\|^2_{L^2}.\]
Decompose the bi-infinite identity $\mathbf{I}\;=\;\mathbf{I}_\ell\;+\;\mathbf{I}_{irr}\;+\mathbf{I}_r$ with 
\begin{equation*}
 \mathbf{I}_\ell(j,j)=\left\{ \begin{array}{cc} 1, & j<1-2n, \\
0, & \hbox{otherwise}, \end{array}\right. \quad 
 \hbox{and} \quad  \mathbf{I}_r(j,j)=\left\{ \begin{array}{cc} 1, & j>2n-1, \\
0, & \hbox{otherwise}. \end{array}\right. 
\end{equation*}
Then, for all $f \in L^2(\mathbb{R})$, by \eqref{eq:our_S} and the Cauchy-Schwarz inequality, we have
\[\begin{array}{l}
\int_{\mathbb{R}^2}\;f(x)\;\Phi(x)^T\;\mathbf{S}\;\Phi(y)\;f(y)\;dx\;dy \;=\;
\int_{\mathbb{R}^2}\;f(x)\;\Phi(x)^T\;\left(\;\mathbf{I}_\ell\;+\;\mathbf{I}_{irr}\;\mathbf{S}\;\mathbf{I}_{irr}\;+\;\mathbf{I}_r\;\right)\;\Phi(y)\;f(y)\;dx\;dy\\ \\
\quad\qquad\qquad\leq\;\max\{C_\ell,C_r\}\;\|f\|^2_{L^2}\;+\;\sum\limits_{j=2-2n}^{2n-2}\;\sum\limits_{k=2-2n}^{2n-2}\;|\mathbf{S}(j,k)|\;
\left|\int_{\mathbb{R}}\;f(x)\;\phi_j(x)\;dx \right| \left|\int_\mathbb{R}\;\phi_k(y)\;f(y)\;dy \right| \\ \\
\quad\qquad\qquad\leq\;\left(\;\max\{C_\ell,C_r\}\;+ (4n-3)\;\|\mathbf{S}_{irr}\|_{\infty}\;\max_{k\in\mathcal{I}_{irr}}\|\phi_k\|^2_{L^2}\;\right)\;\|f\|^2_{L^2}.
\end{array}\]
Part $(ii)$: Using again $\mathbf{I}\;=\;\mathbf{I}_\ell\;+\;\mathbf{I}_{irr}\;+\mathbf{I}_r$ and \eqref{eq:phi_supp}, for every
$f\in L^2(\mathbb{R})$, we get
\[\begin{array}{l}
 \left\| f - 2^j\;\int_{\mathbb{R}} f(x)\;\Phi(2^jx)^T\;\mathbf{S}\;\Phi(2^j\cdot)\;dx \right\|^2_{L^2} \;\leq\; 
 \left\| f\;\chi_{(-\infty,0)} - 2^j\;\int_{-\infty}^0 f(x)\;\Phi_\ell(2^jx)^T\;\Phi_\ell(2^j\cdot)\;dx\right\|^2_{L^2}\;+ \\ \\
\quad+\;\left\| 2^j \int_{\mathbb{R}} f(x)\;\Phi(2^jx)^T\;\mathbf{I}_{irr}\;\mathbf{S}\;\mathbf{I}_{irr}\;\Phi(2^j\cdot)\;dx \right\|^2_{L^2}
\;+\;\left\| f \;\chi_{(0,\infty)} - 2^j \int_{0}^{\infty} f(x)\;\Phi_r(2^jx)^T\;\Phi_r(2^j\cdot)\;dx \right\|^2_{L^2} \\ \\
\quad=:\; \gamma_\ell\;+\;\gamma_{irr}\;+\;\gamma_r.
\end{array}\]
The indices of the non-zero elements of $\mathbf{I}_{irr}\;\mathbf{S}\;\mathbf{I}_{irr}$ belong to the set 
$\mathcal{I}_{irr}\times\mathcal{I}_{irr}$ in  \eqref{eq:I_irr}, thus, 
\[
  \bigcup_{k\in\mathcal{I}_{irr}}\;\supp(\phi_k)=:[a,b], \quad -\infty<a\le b < \infty.
\]
The continuity of $\Phi$ and the Cauchy-Schwarz inequality yield
$$
\gamma_{irr}\;=\;2^{2j} \int_{\frac{[a,b]}{2^j}}\;\left| \int_{\frac{[a,b]}{2^j}}\;f(x)\;\Phi(2^jx)^T\;\mathbf{I}_{irr}\;\mathbf{S}
\;\mathbf{I}_{irr}\;\Phi(2^jy)\;dx \right|^2 dy 
\;\leq\; C\; \|f\|^2_{L^2\left(\frac{[a,b]}{2^j}\right)}
$$
with the constant $C=(b-a)^2 \|\Phi^T\;\mathbf{I}_{irr}\;\mathbf{S}\;\mathbf{I}_{irr}\;\Phi\|^2_{L^\infty}$. Thus, $\gamma_{irr}$
goes to zero as $j$ goes to $\infty$. Moreover, since $f\chi_{(-\infty,0)}$ and $f\chi_{(0,\infty)}$ belong to $L^2(\mathbb{R})$, by the argument 
from the regular case, both  $\gamma_\ell$ and $\gamma_r$ go to zero as $j$ goes to $\infty$. 

\noindent On the other hand, for the sequence 
\[
\mathbf{f}_j\;=\;2^{j/2}\int_\mathbb{R}\;f(x)\;\Phi(2^jx)\;dx, \quad f\in L^2(\mathbb{R}), \quad j\in\mathbb{Z},
\]
with $[a_k,b_k]=\operatorname{supp}(\phi_k)$, $k \in \mathbb{Z}$, by Lemma~\ref{lem:Lemma1} part $(iii)$, we have
\begin{eqnarray*}
\|\mathbf{f}_j\|^2_{\ell^2} \;\leq\; \sum\limits_{k\in\mathbb{Z}}\;\int_{\frac{[a_k,b_k]}{2^j}}\;\left|f(x)\right|^2\;dx\;\int_{\mathbb{R}}\;\left|\phi_k(2^j x)\right|^2\;2^j\;dx 
\;\leq\;  s\; \|f\|^2_{L^2}\;\max_{k\in\mathbb{Z}}\|\phi_k\|^2_{L^2}\;<\; \infty.
\end{eqnarray*}
Thus,  $\mathbf{f}_j \in \ell^2(\mathbb{Z})$ and, consequently, $\mathbf{S}\;\mathbf{f}_j \in\ell^2(\mathbb{Z})$. Hence, by Lemma~\ref{lem:Lemma1} part $(i)$,
we have
\[2^j\;\int_{\mathbb{R}}\;f(x)\;\Phi(2^jx)^T\;\mathbf{S}\;\Phi(2^j\cdot)\;dx\;=\;2^{j/2}\;\mathbf{f}_j^T\;\mathbf{S}\;\Phi(2^j\cdot)\;\in\;\overline{\lspan(\Phi(2^j\cdot))}^{L^2}.
\] 
Due to the MRA structure \eqref{eq:MRA}, we get $\lim_{j\rightarrow -\infty} 2^{j/2}\;\mathbf{f}_j^T\;\mathbf{S}\;\Phi(2^j\cdot) =0$.
Therefore, the claim follows due to the continuity of the inner product.
\end{proof}

\noindent Examples in section~\ref{sec:examples} and  numerical evidence 
for $n=3,\ldots,8$ with different $h_\ell,h_r>0$ (defining the mesh $\mathbf{t}$) lead to the following conjecture.

\begin{conj} \label{conjecture} Let  $\mathbf{S}$ be defined by Algorithm~1, $\mathbf{P}$ and $\mathbf{D}$ as in
section~\ref{sec:DD} and $\mathbf{p}$, $\mathbf{q}_1$ and $\mathbf{q}_2$ as in Proposition~\ref{cor:regDDframe}.
\begin{enumerate}
 \item[$(i)$] For
\begin{equation}\label{eq:R} 
\mathbf{R}=\mathbf{S}\;-\frac{1}{2}\;\mathbf{D}^{1/2}\;\mathbf{P}\;\mathbf{D}^{-1/2}\mathbf{S}\;\mathbf{D}^{-1/2}\;\mathbf{P}^T\;\mathbf{D}^{1/2},
\end{equation}
and  $\mathbf{R}_k$, $k\not\in\mathcal{I}_{irr}$ with entries
\begin{equation}\label{eq:R_reg_block}
\mathbf{R}_k(u,v)=\mathbf{p}\mathbf{p}^T(u-2k,v-2k) +\mathbf{q}_1\mathbf{q}_1^T (u-2k,v-2k)+
\mathbf{q}_2\mathbf{q}_2^T (u-2k,v-2k), \quad u,v\in\mathbb{Z},
\end{equation}
the matrix
\begin{equation} \label{eq:R_irr}
  \mathbf{R}_{irr}=\mathbf{R}\;-\; \frac{1}{2}\sum\limits_{k\not\in\mathcal{I}_{irr}}\;\mathbf{R}_k
	\end{equation}
 is positive semi-definite.
\item[$(ii)$] For all $\alpha\in\{0,\dots,n-1\}$, the bi-infinite vectors $\mathbf{c}_\alpha$ in \eqref{eq:DD_rep_renorm} and the
moments
$\mathbf{m}_\alpha$ in \eqref{eq:moments} satisfy 
\begin{equation}\label{eq:mom_to_coeff}
                 \mathbf{S}\;\mathbf{m}_\alpha\;=\;\mathbf{c}_\alpha.
\end{equation}
\end{enumerate}
\end{conj}

\begin{rem} Note that the requirement that $\mathbf{R}_{irr}$ is positive semi-definite is stronger than the positive semi-definedness
of $\mathbf{R}$, which is one of the sufficient conditions in \cite[Theorem 4.2]{MR2110512} that ensures the existence of the 
wavelet tight frame
\begin{equation}\label{def:Psi}
 \Psi_j=2^{j/2}\mathbf{Q}^T \Phi(2^j\cdot), \quad \mathbf{R}\;=\;\mathbf{Q}\;\mathbf{Q}^T,\quad j\in\mathbb{N}
\end{equation}
Here indeed we also get an explicit form 
\[
 \mathbf{Q}\;=\;\left[\begin{array}{ccccc} \dots & \mathbf{Q}_{\min(\mathcal{I}_{irr})-1} &\mathbf{Q}_{irr}& 
\mathbf{Q}_{\max(\mathcal{I}_{irr})+1}&\dots \end{array}\right]
\]
with $\mathbf{R}_{irr}=\mathbf{Q}_{irr}\mathbf{Q}_{irr}^T$ and $\mathbf{R}_{k}=\mathbf{Q}_{k}\mathbf{Q}_{k}^T$ for 
$k \not\in\mathcal{I}_{irr}$.
\end{rem}

\noindent Apart from Examples in section~\ref{sec:examples} and  numerical evidence 
for $n=3,\ldots,8$, there are other facts that support Conjecture~\ref{conjecture}.

\begin{rem} \label{rem:time_domain_Chui} 
$(i)$ In the regular case Conjecture~\ref{conjecture} is well known. Indeeed, the characterization in 
\cite[Theorem 4.2]{MR2110512} is also valid and the construction in Proposition~\ref{cor:regDDframe} is one particular factorization 
of $\mathbf{R}$ in \eqref{eq:R}  with $\mathbf{S}=\mathbf{I}$. 
Indeed, \eqref{eq:UEP} is equivalent to
\[\left\{\begin{array}{l}
\;[ e^{i2k\pi\omega}\;:\;k\in\mathbb{Z}]^T\;\left(\;\frac{1}{2}\mathbf{p}\mathbf{p}^T\;+\;\frac{1}{2}\sum\limits_{j=1}^J\mathbf{q}_j\mathbf{q}_j^T\;\right)\;[\;e^{-i2k\pi\omega}\;:\;k\in\mathbb{Z}] =  2, \\
\;[ e^{i2k\pi\omega}\;:\;k\in\mathbb{Z}]^T\;\left(\;\frac{1}{2}\mathbf{p}\mathbf{p}^T\;+\;\frac{1}{2}\sum\limits_{j=1}^I\mathbf{q}_j\mathbf{q}_j^T\;\right) [(-1)^k e^{i2k\pi\omega}\;:\;k\in\mathbb{Z}]=0,
\end{array}\right. \quad \omega \in [0,1).
\]
The identity \eqref{eq:R_irr}, due to $\mathbf{t}=\mathbb{Z}$, becomes
\begin{equation*}
   \mathbf{R}_{irr}=\mathbf{R}-\frac{1}{2}\sum\limits_{k\in\mathbb{Z}}\;\mathbf{R}_k=0.
\end{equation*}
Moreover, in the regular case, by \cite{MR2110512} and Proposition~\ref{cor:regDDframe}, 
we have
\begin{equation}\label{eq:reg_m_c}
\mathbf{m}_\alpha\;=\;\mathbf{c}_\alpha,\quad \alpha\in\{0,\dots,n-1\}.
\end{equation}

\noindent $(ii)$ In the irregular case,  due to the special structure of $\mathbf{S}$ and \eqref{eq:reg_m_c} from the regular case, 
the identity \eqref{eq:mom_to_coeff} reduces to an identity for certain finite matrices. For $\mathbf{m}_\alpha$ in \eqref{eq:moments} and $\mathbf{c}_\alpha$ in \eqref{eq:DD_rep_renorm}, $\alpha=0,\ldots,n-1$, define 
\[\mathbf{M}=\left[[\mathbf{m}_0(k)]_{k\in\mathcal{I}_{irr}}|\dots|[\mathbf{m}_{n-1}(k)]_{k\in\mathcal{I}_{irr}}\right]\quad\textrm{and}\quad\mathbf{C}=\left[[\mathbf{c}_0(k)]_{k\in\mathcal{I}_{irr}}|\dots|[\mathbf{c}_{n-1}(k)]_{k\in\mathcal{I}_{irr}}\right].\]
Then the irregular part of \eqref{eq:mom_to_coeff} becomes
$\mathbf{S}_{irr} \mathbf{M} =\mathbf{C}$, which implies
\[ \mathbf{M}^T\;\mathbf{C}\;=\;\mathbf{M}^T\;\mathbf{S}_{irr}\;\mathbf{M}.\]
Thus, for \eqref{eq:mom_to_coeff} to hold the matrix $\mathbf{M}^T\;\mathbf{C}$ must be symmetric. 
Indeed, for every $\alpha,\beta\in\{0,\dots,n-1\}$, by \eqref{eq:DD_rep_renorm}, we get
\[\begin{array}{l}
0\;=\;x^\alpha\Phi^T\;\mathbf{c}_\beta-\mathbf{c}_\alpha^T\;\Phi x^\beta 
\;=\; [x^\alpha\phi_k(x)]^T_{k\in\mathcal{I}_{irr}}\;[\mathbf{c}_{\beta}(k)]_{k\in\mathcal{I}_{irr}}\;
-\;[\mathbf{c}_{\alpha}(k)]_{k\in\mathcal{I}_{irr}}^T\;[x^\beta\phi_k(x)]_{k\in\mathcal{I}_{irr}} \\ \\
\qquad\qquad\qquad\qquad\qquad\qquad\qquad\qquad\qquad\qquad\;+\; \sum\limits_{k\not\in\mathcal{I}_{irr}} \left(\mathbf{c}_\beta(k) \; x^\alpha\phi_k(x)-\mathbf{c}_\alpha(k)x^\beta\phi_k(x) \right),
\quad x \in \mathbb{R}.
\end{array}\]
Integrating both sides of the above identity and using \eqref{eq:reg_m_c} for the summation over $k\not\in\mathcal{I}_{irr}$, we obtain
\[\begin{array}{l}
 (\mathbf{M}^T\mathbf{C})(\alpha,\beta)-(\mathbf{M}^T \mathbf{C})(\beta,\alpha)c\;=\; \\ \\
\qquad\qquad\;=\; [\mathbf{m}_\alpha(k)]_{k\in\mathcal{I}_{irr}}^T \; [\mathbf{c}_\beta(k)]_{k\in\mathcal{I}_{irr}} \;-\;[\mathbf{c}_\alpha(k)]_{k\in\mathcal{I}_{irr}}^T\; [\mathbf{m}_\beta(k)]_{k\in\mathcal{I}_{irr}}=0, \qquad \alpha,\beta\in\{0,\dots,n-1\}.
\end{array}\]
We strongly believe  that part $(ii)$ of Conjecture~\ref{conjecture} is due to some special, intriguing property of the Dubuc-Deslauriers schemes.
\end{rem}

\noindent 
Furthermore, Lemma \ref{lem:Lemma2} and Conjecture~\ref{conjecture} part $(ii)$ are sufficient to ensure the higher number of vanishing moments
of framelets in \eqref{def:Psi} without additional assumptions \cite[Assumption 3, Corollary 4.8]{MR2110512} on the existence of the anti-difference schemes. Proposition~\ref{prop:eigenvector} shows that part $(i)$ of Conjecture~\ref{conjecture} is sufficient to guarantee $n$ vanishing moments for Dubuc-Deslauriers framelets in \eqref{def:Psi}. 

\begin{prop} \label{prop:eigenvector}
Let $\mathbf{P}$ be the subdivision matrix of a convergent Dubuc-Deslauriers $2n$-point semi-regular scheme. If $\;\mathbf{R}_{irr}$ in
\eqref{eq:R_irr} is positive semi-definite, then the framelets $\Psi$ in \eqref{def:Psi} satisfy
$$
 \int_{\mathbb{R}}\;\Psi_j(x)\; x^\alpha\;dx\;=\;\mathbf{0}, \quad  \alpha\in\{0,\dots,n-1\},\;j\in\mathbb{N}.
$$
\end{prop}
\begin{proof} Let $\alpha \in\{0,\dots,n-1\}$. We start by proving that the moments $\mathbf{m}_\alpha$ in \eqref{eq:moments} and 
$\mathbf{c}_\alpha$ in \eqref{eq:DD_rep} satisfy
\begin{equation}\label{eq:m_c_eig}
\mathbf{m}_\alpha^T\;(\mathbf{D}/2)^{1/2}\mathbf{P}\mathbf{D}^{-1/2}\;=\;2^{\alpha+1/2}\mathbf{m}_\alpha^T\qquad\textrm{and}\qquad(\mathbf{D}/2)^{1/2}\mathbf{P}\mathbf{D}^{-1/2}\mathbf{c}_\alpha\;=\;2^{-(\alpha+1/2)}\mathbf{c}_\alpha
\end{equation}
with $\mathbf{D}$ in \eqref{eq:D_mat}.
From \eqref{eq:moments}, using the refinement equation \eqref{eq:ref_eq_renorm}, we obtain
\begin{eqnarray*}
2^{\alpha+1/2}\mathbf{m}_\alpha^T&=&2^{\alpha+1/2}\;\int_\mathbb{R}\;x^\alpha\;\Phi^T(x)\;dx\;=
\; 2^{\alpha+1/2}\;\int_\mathbb{R}\;x^\alpha\;\Phi^T(2x)\;dx\;\mathbf{D}^{1/2}\;\mathbf{P}\;\mathbf{D}^{-1/2}\\ \\
&=&\frac{1}{\sqrt{2}}\;\int_\mathbb{R}\;y^\alpha\;\Phi^T(y)\;dy\;\mathbf{D}^{1/2}\;\mathbf{P}\;\mathbf{D}^{-1/2} \;=\; \mathbf{m}_\alpha^T\;(\mathbf{D}/2)^{1/2}\mathbf{P}\mathbf{D}^{-1/2}.
\end{eqnarray*}
To show the second identity in \eqref{eq:m_c_eig}, first note that, by the construction of $\mathbf{P}$ in Section~\ref{sec:DD}, 
we have that $\mathbf{P}\mathbf{t}^\alpha=2^{-\alpha}\mathbf{t}^\alpha$ for $\mathbf{t}$ in \eqref{eq:t}. Thus, 
due to \eqref{eq:DD_rep_renorm}, we get
 \begin{eqnarray*}
2^{-(\alpha+1/2)}\mathbf{c}_\alpha=2^{-(\alpha+1/2)}\;\mathbf{D}^{1/2}\;\mathbf{t}^\alpha\;=\; \frac{1}{\sqrt{2}}\;\mathbf{D}^{1/2}\;\mathbf{P}\;\mathbf{t}^\alpha
=(\mathbf{D}/2)^{1/2}\mathbf{P}\mathbf{D}^{-1/2}\;\mathbf{D}^{1/2}\;\mathbf{t}^\alpha\;=\;(\mathbf{D}/2)^{1/2}\mathbf{P}\mathbf{D}^{-1/2}\;\mathbf{c}_\alpha.
\end{eqnarray*}
Next, by \eqref{eq:R} and the first identity in \eqref{eq:m_c_eig}, we get
$$
\begin{array}{rcl}
 \mathbf{R}\;\mathbf{m}_\alpha=\mathbf{S}\;\mathbf{m}_\alpha\;-\;\frac{1}{2}\;\mathbf{D}^{1/2}\;\mathbf{P}\;\mathbf{D}^{-1/2}\mathbf{S}\;\mathbf{D}^{-1/2}\;\mathbf{P}^T\;\mathbf{D}^{1/2}\mathbf{m}_\alpha
= \mathbf{S}\;\mathbf{m}_\alpha\;-2^{\alpha+1/2}\;(\mathbf{D}/2)^{1/2}\;\mathbf{P}\;\mathbf{D}^{-1/2}\mathbf{S}\;\mathbf{m}_\alpha.
\end{array}
$$
Thus, by \eqref{eq:mom_to_coeff} and the second identity in \eqref{eq:m_c_eig}, $\mathbf{R} \mathbf{m}_\alpha=0$.
Since, by Proposition~\ref{cor:regDDframe} and \eqref{eq:R_reg_block}, we have $\mathbf{R}_k\mathbf{m}_\alpha=0$, $k\not\in\mathcal{I}_{irr}$, due to \eqref{eq:R_irr}, we obtain 
$$
 \mathbf{R}_{irr}\;\mathbf{m}_\alpha\;=\;\mathbf{R}\;\mathbf{m}_\alpha\;-\frac{1}{2}\;\sum\limits_{k\not\in\mathcal{I}_{irr}}\;\mathbf{R}_k\;\mathbf{m}_\alpha\;=\;\mathbf{0}.
$$
Therefore, with $\mathbf{Q}$ in \eqref{def:Psi}, we have
\[\int_{\mathbb{R}}\;\Psi_j(x)\; x^\alpha\;dx\;=\;2^{j/2}\;\mathbf{Q}^T\;\int_{\mathbb{R}}\;\Phi(2^jx)\; x^\alpha\;dx\;=\;2^{-j(\alpha+1/2)}\;\mathbf{Q}^T\;\mathbf{m}_\alpha\;=\;\mathbf{0}.\]
\end{proof}

\begin{rem} Note that while the first identity in \eqref{eq:m_c_eig} does not rely on any property of Dubuc-Deslauriers schemes, the second one is particular to subdivision that reproduces polynomial up to degree $n-1$. This justifies our choice of the Dubuc-Deslauriers schemes, 
which posses the best ratio between polynomial reproduction and the support length of their basic limit functions.
\end{rem}

\begin{rem}
If one chooses $\widetilde{n}<n$ columns of $\mathbf{O}$ in \eqref{eq:our_S}, then the corresponding matrix $\mathbf{S}$ would
generate a wavelet tight frame with $\widetilde{n}$ vanishing moments. Since it is not possible to get more than $n$ vanishing moments
in the regular case, the choice $\widetilde{n}=n$ is optimal in the semi-regular case.
\end{rem}

\section{Examples} \label{sec:examples}

\noindent In subsections~\ref{subsec:n_1} and \ref{subsec:n_2}, we present two simple examples illustrating the construction in section~\ref{sec:semi_reg} for $n=1$ and
$n=2$, respectively. The small bandwidth  of the corresponding subdivision matrices $\mathbf{P}$ allows for exact computations in terms of
the mesh parameter. Without loss of generality, after a suitable renormalization, we consider the mesh $\mathbf{t}$ with $h_\ell=1$ and $h_r=h$, $h>0$. In the case $n=1$, the Dubuc-Deslauriers $2$-point scheme corresponds to the linear B-spline scheme. The case $n=2$, is more interesting and involved due to the high complexity of the entries of the corresponding matrices. For these two examples we are able to prove both parts of 
Conjecture~\ref{conjecture}. 

For the interested reader, the irregular filters $\mathbf{Q}_{irr}$ for $n=2,3,4,5$ and several values of $h_r>0$ are available in \cite{filters}.

\subsection{Case $n=1$: linear B-spline scheme} \label{subsec:n_1}

\noindent In the regular case, i.e. $h_\ell=h_r=1$, the linear B-spline scheme is defined by the mask \[[\mathbf{p}(k)\ : \ k=-1,0,1]=
\left[\begin{array}{ccc}\frac{1}{2} & 1 & \frac{1}{2} \end{array} \right].\] By 
Proposition~\ref{cor:regDDframe}, with $[\mathbf{d}(k)\ :\ k=-1,0]=\left[\begin{array}{cc}1& 1 \end{array} \right]$, we get the well-known
\[
  [\mathbf{q}_1(k)\ : \ k=-1,0,1]=\frac{1}{\sqrt{2}} \left[\begin{array}{ccc} 1 & 0 &-1 \end{array} \right] \quad \textrm{and} \quad 
	[\mathbf{q}_2(k)\ : \ k=-1,0,1]=\frac{1}{2}\left[\begin{array}{ccc} -1&2&-1 \end{array} \right].
\]
In the semi-regular case,  the subdivision matrix $\mathbf{P}$ does not depend on $h$ and is the $2$-slanted matrix with columns
determined by $\mathbf{p}^T$. The corresponding basic limit functions $\{\varphi_k\;:\;k\in\mathbb{Z}\}$ of the subdivision scheme, 
the so-called \emph{hat functions},
\[
\varphi_k(x)=\left\{\begin{array}{cl}
\frac{x-\mathbf{t}(k-1)}{\mathbf{t}(k)-\mathbf{t}(k-1)},& x\in[\mathbf{t}(k-1),\mathbf{t}(k)], \\ \\
\frac{\mathbf{t}(k+1)-x}{\mathbf{t}(k+1)-\mathbf{t}(k)},& x\in[\mathbf{t}(k),\mathbf{t}(k+1)], \\ \\
0,& \textrm{otherwise},
\end{array}\right.\quad
\textrm{satisfy}
\quad\int_\mathbb{R}\;\varphi_k(x)\;dx\;=\;\left\{\begin{array}{cl}
1, & k <0,\\ \\
\frac{1+h}{2},& k=0, \\ \\
h, &  k>0.
\end{array}\right.\]
Thus, the entries of $\mathbf{D}$ in \eqref{eq:D_mat} are well defined for every $h>0$ and, the first moments of the scaling functions, 
by \eqref{eq:D_mat}, 
satisfy  
\[
 \mathbf{m}_0=\mathbf{D}^{1/2} \mathbf{1}=\left[ \begin{array}{ccccccccc} \dots & 1& \dots & 1 & \sqrt{\frac{1+h}{2}} & \sqrt{h} & \dots & \sqrt{h}&\dots \end{array}\right].
\]
The {\bf Algorithm 1} computes $\mathbf{S}_{irr}=1$, thus, by \eqref{eq:DD_rep_renorm}, part $(ii)$ of Conjecture~\ref{conjecture} is true. Moreover, by \eqref{eq:R_irr}, we get
\[\mathbf{R}_{irr}\;=\;{\scriptsize\begin{bmatrix}
\frac{2h+1}{4(h+1)} & -\frac{\sqrt{2}}{4\sqrt{h+1}} & -\frac{\sqrt{h}}{4(h+1)} \\ \\
-\frac{\sqrt{2}}{4\sqrt{h+1}}  & \frac{1}{2} & -\frac{\sqrt{2h}}{4\sqrt{h+1}} \\ \\
-\frac{\sqrt{h}}{4(h+1)} & -\frac{\sqrt{2h}}{4\sqrt{h+1}} & \frac{h+2}{4(h+1)}
\end{bmatrix}}\;=\;\mathbf{Q}_{irr}\mathbf{Q}_{irr}^T\quad\textrm{with}\quad\mathbf{Q}_{irr}\;=\;{\scriptsize\left[\begin{array}{cc} -\frac{1}{2\,\sqrt{h+1}} & \frac{\sqrt{2h}}{2\,\sqrt{h+1}}\\ \\
 \frac{\sqrt{2}}{2} & 0\\ \\
-\frac{\sqrt{h}}{2\,\sqrt{h+1}} & -\frac{\sqrt{2}}{2\,\sqrt{h+1}} \end{array}\right]},\] 
therefore, part $(i)$ of Conjecture~\ref{conjecture} is also true. Proposition~\ref{prop:eigenvector} guarantees  that the irregular framelets in \eqref{def:Psi} have one vanishing moment.

\subsection{Case $n=2$: Dubuc-Deslauriers $4$-point scheme} \label{subsec:n_2}

\noindent For $n=2$, which is also a special case of the family of schemes constructed in \cite{MR2805717,MR2855428}, we obtain, via the construction in section~\ref{sec:DD}, the regular columns of $\mathbf{P}$ as shifts of the regular mask 
\[[\mathbf{p}(k)\ : \ k=-3,\dots,3]=\frac{1}{16}\left[\begin{array}{ccccccc}-1&0&9&16&9&0&-1\end{array}\right]\] 
and the five irregular columns of $\mathbf{P}$ are given by
\[[\mathbf{P}(m,k)]_{\substack{-7\leq m\leq7 \\k\in\mathcal{I}_{irr}}}\;=\;{\scriptsize\begin{bmatrix}
                   -1/16&                              &                    &                          &                        \\ 
                       0&                              &                    &                          &                        \\ 
                    9/16&                          -1/16&                    &                          &                        \\ 
                       1&                              0&                    &                          &                        \\ 
                    9/16&                           9/16&                -1/16&                          &                        \\ 
                       0&                              1&                    0&                         &                        \\ 
 -\frac{2h + 1}{16(h + 2)}&      \frac{3(2h + 1)}{8(h + 1)}& \frac{3(2h + 1)}{16h}& -\frac{3}{8h(h+1)(h+2)}&                       \\ 
                       &                              0&                    1&                          0&                       \\ 
                       & -\frac{3h^3}{8(2h + 1)(h + 1)}&       \frac{3(h+2)}{16}&    \frac{3(h + 2)}{8(h + 1)}& -\frac{h + 2}{16(2h + 1)} \\ 
                       &                              &                    0&                          1&                       0 \\ 
                       &                              &                -1/16&                       9/16&                    9/16 \\ 
                       &                              &                    &                          0&                       1 \\ 
                       &                              &                    &                      -1/16&                    9/16 \\ 
                       &                              &                    &                          &                       0 \\ 
                       &                              &                    &                          &                   -1/16
\end{bmatrix}}.
\]
Proposition~\ref{cor:regDDframe} with 
$[\mathbf{d}(k)\ : \ k=-3,\dots,0]=\frac{1}{4}\left[\begin{array}{cccc}1+\sqrt{3}&3+\sqrt{3}&3-\sqrt{3}&1-\sqrt{3}\end{array}\right]$,  yields
\[[\mathbf{q}_1(k)\ : \ k=-3,\dots,3] \;=\; \frac{\sqrt{2}}{16}\;\begin{bmatrix} \sqrt{3}-2& 0& 6-\sqrt{3}&   0& -6-\sqrt{3}& 0& \sqrt{3}+2\end{bmatrix},\]
\[[\mathbf{q}_2(k)\ : \ k=-3,\dots,3] \;=\; \frac{1}{16}\;\begin{bmatrix}   1& 0& -9& 16&    -9& 0&   1\end{bmatrix}.\]
Applying \cite{MR1302255} in the regular part of the mesh, we obtain $\mathbf{m}_0(k)=1$, $k<-2$, and
$\mathbf{m}_0(k)=\sqrt{h}$, $k>2$, and, by Lemma~\ref{lem:Lemma2} and Theorem~\ref{thm:semireg_prod_mat} 
applied to the corresponding basic limit functions, we get
\begin{eqnarray*} 
\mathbf{m}_0(-2)&=& \sqrt{\frac{1}{120}\left(h-\frac{3}{2}\right)^2+\frac{479}{480}},   \quad
\mathbf{m}_0(-1)= \sqrt{\frac{(7-2h)(h+2)}{15}}, \quad 
\mathbf{m}_0(0)= \sqrt{\frac{(h + 1)^3}{8h}}, \\   
\mathbf{m}_0(1)&=& \sqrt{\frac{(7h - 2)(2h + 1)}{15h}} \quad \hbox{and} \quad
\mathbf{m}_0(2)=  \sqrt{\frac{122}{120h}\left(h-\frac{3}{244}\right)^2+\frac{479}{58560h}}.
\end{eqnarray*}
The expressions for $\mathbf{m}_0(-1)$ and $\mathbf{m}_0(1)$ imply that $\mathbf{D}$  in \eqref{eq:D_mat} is positive definite 
if and only if $h\in\left(\frac{2}{7},\frac{7}{2}\right)$. 
Thus, the frame construction in Section~\ref{sec:DD} is not valid for other $h$. Moreover, for the second moment we have
$$
 \mathbf{m}_1(k)=k, \quad k<-2, \quad \hbox{and}\quad \mathbf{m}_1(k)=k\sqrt{h}, \quad k>2,
$$
and in the irregular part
\begin{equation*} \label{eq:1st_mom_4pt}
\begin{array}{l}
[\mathbf{m}_1(k)\ : \ k=-2,\dots,2]\;= \\ \\
		\qquad\diag([\mathbf{m}_0(k)\ : \ k=-2,\dots,2])^{-1}\;
		\left[\;\frac{h^3-3h^2+7h-1205}{600};\;
            -\frac{(h+2)(4h^2-14h+35)}{75};\; \dots\right. \\ \\
            \qquad\left.\dots\;\frac{(h + 1)(h-1)(31h^2+40h+31)}{600h};\;
            \frac{(2h + 1)(35h^2-14h+4)}{75h};\;
            \frac{1205h^3-7h^2+3h-1}{600h}\right].
\end{array}
\end{equation*}
Next, we construct $\mathbf{S}_{irr}$ to check the validity of Conjecture~\ref{conjecture}. The entries of $\mathbf{S}_{irr}$ depend in
an intricate way on the parameter $h$, thus, we work with $\widetilde{\mathbf{S}}_{irr}$ instead, where,  
for
\[\alpha\;=\;\frac{5(h+1)}{2},\qquad\beta\;=\;\frac{37(h^2-1)}{12},\qquad\gamma\;=\;\frac{5(h+1)^3(431h^2+938h+431)}{288},\]
we have
\[\mathbf{S}_{irr}\;=\;\frac{1}{\alpha\gamma}\diag([\mathbf{m}_0(k)\ : \ k=-2,\dots,2])\;\widetilde{\mathbf{S}}_{irr}\;\diag([\mathbf{m}_0(k)\ : \ k=-2,\dots,2]).\]
with\\ $ $\\
\verb+ +{\footnotesize
$\widetilde{\mathbf{S}}_{irr}\;=\;(\;\alpha\;\beta^2\;+\;\gamma\;)\;[(\mathbf{1}\mathbf{1}^T)(m,k)]_{-2\leq m,k\leq2}\;+\;\alpha^3\;[(\mathbf{t}\mathbf{t}^T)(m,k)]_{-2\leq m,k\leq2}\;-\;\alpha^2\beta\;\left[\big(\mathbf{1}\mathbf{t}^T+\mathbf{t}\mathbf{1}^T\big)(m,k)\right]_{-2\leq m,k\leq2}$ \\ $ $ \\
\verb+  +$=\;\frac{25(h + 1)^3}{48}\;$}
{\scriptsize$\left[\begin{array}{ccccc} 60\,h^2+88\,h+32 & 60\,h^2+51\,h+9 & 60\,h^2+14\,h-14 & 23\,h^2-9\,h-14 & -14\,h^2-32\,h-14\\ 60\,h^2+51\,h+9 & 60\,h^2+14\,h+16 & 60\,h^2-23\,h+23 & 23\,h^2-16\,h+23 & -14\,h^2-9\,h+23\\ 60\,h^2+14\,h-14 & 60\,h^2-23\,h+23 & 60\,h^2-60\,h+60 & 23\,h^2-23\,h+60 & -14\,h^2+14\,h+60\\ 23\,h^2-9\,h-14 & 23\,h^2-16\,h+23 & 23\,h^2-23\,h+60 & 16\,h^2+14\,h+60 & 9\,h^2+51\,h+60\\ -14\,h^2-32\,h-14 & -14\,h^2-9\,h+23 & -14\,h^2+14\,h+60 & 9\,h^2+51\,h+60 & 32\,h^2+88\,h+60 \end{array}\right] $.}\\ $ $\\
Note that part $(ii)$  of Conjecture~\ref{conjecture} is equivalent to the system 
\[\left\{\begin{array}{rcl}
\widetilde{\mathbf{S}}_{irr}\;\diag([\mathbf{m}_0(k)\ : \ k=-2,\dots,2])\;[\mathbf{m}_0(k)\ : \ k=-2,\dots,2]&=&\alpha\;\gamma\;[\mathbf{1}(k)\ : \ k=-2,\dots,2] \\ \\
\widetilde{\mathbf{S}}_{irr}\;\diag([\mathbf{m}_0(k)\ : \ k=-2,\dots,2])\;[\mathbf{m}_1(k)\ : \ k=-2,\dots,2]&=&\alpha\;\gamma\;[\mathbf{t}(k)\ : \ k=-2,\dots,2]
\end{array}\right.\]
of polynomial equations, whose validity we checked with the help of MATLAB symbolic tool. Due to $\alpha,\gamma>0$ for $h\in\left(\frac{2}{7},\frac{7}{2}\right)$, part $(i)$ of Conjecture~\ref{conjecture}
is equivalent to checking that
\[
 \widetilde{\mathbf{R}}_{irr}\;=\;\alpha\;\gamma\;h\;\diag(\mathbf{m}_0)^{-1}\;\mathbf{R}_{irr}\;\diag(\mathbf{m}_0)^{-1}
\]
is positive semi-definite. The renormalization leads to $\widetilde{\mathbf{R}}_{irr}$ with polynomial entries and allows for
symbolic manipulations. Indeed, this way, the generalized Sylvester criterion, confirms that $\widetilde{\mathbf{R}}_{irr}$ is
positive semi-definite for $h\in\left(\frac{2}{7},\frac{7}{2}\right)$. In Figure \ref{fig:4pt} one can see the framelets corrisponding to a possible factorization of $\mathbf{R}$ with $h=2$.

\begin{rem}
The value $2/7 \approx 0.2857$ resembles the corresponding critical value in \cite{MR2954376,MR2832719} computed for the irregular knot insertion for the $4$-point scheme. Below this critical value the scheme loses regularity. This fact makes the restriction on the range of the stepsize $h$ less surprising in this case. 
\end{rem}

\section{Conclusions}

We presented a method for constructing wavelet tight frames associated with the Dubuc-Deslauriers family of semi-regular interpolatory subdivision schemes, from their convergence analysis, to the choice of a suitable approximation of the corresponding Gramian matrix. This is the first step towards developing a practical tool for regularity analysis of wide classes of semi-regular subdivision schemes. 
There are several prominent tools for regularity analysis in the regular setting, e.g. difference operator \cite{MR1079033,MR2008967}, joint spectral radius \cite{MR3606458,MR1142737} and wavelet techniques \cite{MR1162107,MR1228209}. In the irregular setting, the available methods for regularity analysis of subdivision schemes are either based on difference operator techniques in the interpolatory case \cite{MR1687779}, or, for semi-regular schemes, on local eigenvalue analysis technique from \cite{MR2415757,Warren:2001:SMG:580358}. The next step in this direction will be to extend standard wavelet techniques to the wavelet tight frames constructed here for the analysis of wide classes of semi-regular subdivision schemes. The simplicity of our construction may also have a strength in other practical applications.

\newpage
\begin{figure}[htbp]
\centering
\includegraphics[width=\textwidth]{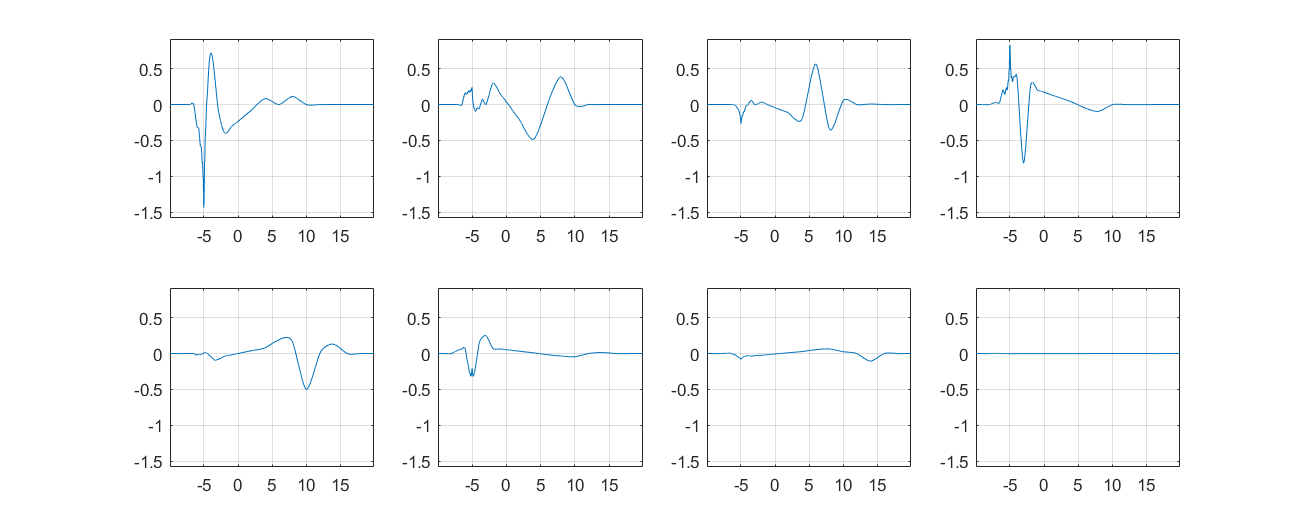}
\[\mathbf{Q}_{irr}\;=\;   
   {\scriptsize\left[\begin{array}{rrrrrrrr}
    0.0000&    0.0000&   -0.0000&    0.0311&    0.0009&    0.0481&   -0.0048&   -0.0037\\
   -0.0000&   -0.0000&   -0.0000&    0.0000&    0.0000&    0.0000&    0.0000&   -0.0000\\
   -0.0000&   -0.0000&    0.0000&   -0.2085&    0.0029&   -0.5363&    0.0452&    0.0005\\
    0.7656&    0.2646&   -0.1091&    0.3036&   -0.0790&    0.0978&    0.0404&    0.0008\\
   -0.0000&   -0.0000&    0.0000&   -0.8142&   -0.0832&    0.2487&    0.0301&    0.0007\\
   -0.4246&    0.2412&   -0.0333&    0.2374&   -0.0390&    0.0757&    0.0231&    0.0007\\
   -0.2950&    0.1101&    0.0041&    0.1827&   -0.0170&    0.0578&    0.0130&    0.0005\\
   -0.3055&    0.0064&    0.0552&    0.2222&    0.0012&    0.0696&    0.0076&    0.0007\\
   -0.0647&   -0.3271&    0.1672&    0.1482&    0.0578&    0.0446&   -0.0161&    0.0005\\
    0.2038&   -0.6632&    0.2752&    0.0547&    0.1147&    0.0134&   -0.0406&    0.0003\\
   -0.0000&    0.0000&   -0.7972&   -0.0663&    0.2687&   -0.0219&   -0.0766&    0.0001\\
    0.0863&    0.5593&    0.4887&   -0.1329&    0.2274&   -0.0491&   -0.0893&   -0.0001\\
   -0.0000&    0.0000&   -0.0765&   -0.0034&   -0.7045&   -0.0610&   -0.0364&   -0.0003\\
    0.0000&    0.0000&     0.0000&    0.0000&    0.0000&    0.0000&    0.0000&    0.0000\\
   -0.0000&    0.0000&   -0.0109&   -0.0015&    0.1847&    0.0201&    0.1492&   -0.0008\\
\end{array}\right]}\]
\caption{Irregular framelets for the Dubuc-Deslauriers $4$-point scheme with $h=2$ corresponding, from left to right, to the columns of a possible factorization of $\mathbf{R}_{irr}\;=\;\mathbf{Q}_{irr}\;\mathbf{Q}_{irr}^T$.}
\label{fig:4pt}
\end{figure}

\section*{Acknowledgements}

This research has been accomplished within RITA (Research ITalian network on Approximation). The author would also like to thank M. Charina, C. Conti, L. Romani and J. St\"ockler for the fruitful discussions and suggestions and the anonymous reviewers for the comments which improved this work considerably.

\section*{References}

\bibliographystyle{siam}
\bibliography{SMART.bbl}

\end{document}